\documentclass[12pt]{amsart}
\title[On a family of mild functions]{On a family of mild functions}
\author[S. Van Hille]{Siegfried Van Hille}
\address{KU Leuven, Celestijnenlaan 200B, 3001 Leu\-ven, Bel\-gium}
\email{siegfried.vanhille@kuleuven.be}

\usepackage[margin=3cm]{geometry}
\usepackage{amsfonts}
\usepackage{amsmath}
\usepackage{amsthm}
\usepackage{amssymb}
\usepackage{graphicx}
\usepackage{amsopn}
\usepackage{enumitem}
\usepackage{xcolor}
\usepackage{tikz}
\usetikzlibrary{matrix}
\usepackage[colorlinks]{hyperref}
\hypersetup{
	linkcolor = {red},
	citecolor = {blue}
}

\DeclareMathOperator{\R}{\mathbb{R}}

\DeclareMathOperator{\N}{\mathbb{N}}

\DeclareMathOperator{\im}{\text{Im}}

\theoremstyle{definition}
\newtheorem{definition}{Definition}[section]

\theoremstyle{plain}
\newtheorem{theorem}{Theorem}[section]

\newtheorem{lemma}[theorem]{Lemma}
\newtheorem{proposition}[theorem]{Proposition}
\newtheorem*{theorem*}{Theorem}
\newtheorem*{corollary*}{Corollary}
\newtheorem*{proposition*}{Proposition}

\makeatletter
\@namedef{subjclassname@2020}{%
  \textup{2020} Mathematics Subject Classification}
\makeatother

\begin{document}
\thanks{The author is partially supported by KU Leuven grant IF C14/17/083.}

\maketitle
\begin{abstract}
We prove that the function $P_\alpha(x) = \exp(1-x^{-\alpha})$ with $\alpha >0$ is $1/\alpha$-mild. We apply this result to obtain a uniform $1/\alpha$-mild parametrization of the family of curves $\{ xy = \epsilon^2 \mid (x,y) \in (0,1)^2 \}$ for $\epsilon \in (0,1)$,  which does not have a uniform $0$-mild parametrization by work of Yomdin. More generally we can parametrize families of power-subanalytic curves. This improves a result of Benjamini and Novikov that gives a $2$-mild parametrization.
\end{abstract}

\section{Introduction}
We study the function 
\[
P_\alpha : (0,1) \to (0,1) : x \mapsto \exp(1-x^{-\alpha})
\]
with $\alpha > 0$ in the context of counting rational points. In \cite{countpointsexp}, Pila gave an upper bound on the number of rational points of bounded height on a certain exponential-algebraic surface by parametrizing it using $P_\alpha$. The bound on the derivatives of $P_\alpha$, or more generally of mild functions, are suited to bound the number of rational points on the set that it parametrizes, see \cite[Corollary 3.3]{countpointsexp}. The constant $C$ in the definition of a mild function (see Definition \ref{defmild}) plays an important role in this bound. 

We will show that for the map $P_\alpha$, one can take $C = 1/\alpha$. This is sharper than the result of Pila, where he shows $C = 1+1/\alpha$ (\cite[Proposition 2.8.]{countpointsexp}) and has the advantage of being close to zero for $\alpha$ large. The main obstacle in bounding the derivatives is the polynomial in $1/x$ that appears in the computations, which is rather large.

The reason why the map $P_\alpha$ is interesting in this field, is because some functions become mild after composing with it. This has been applied successfully for $\alpha = 1$ in \cite{ccs} to construct $2$-mild parametrizations of families of subanalytic sets. In Section \ref{sec3} we show our main results. We prove such properties for some classes of functions which appear in \cite{unif} in Theorem \ref{abmmild}, from which the result on parametrizations of curves follows in Theorem \ref{curvepara}. In both results we obtain $C = 1/\alpha$ for $\alpha > 0$.

If one considers a fixed subanalytic set,  there are model theoretic results that imply that it has a $0$-mild parametrization (see for instance \cite{mildpara} for subanalytic sets). However, Yomdin has shown in \cite{yomdin2} that any analytic parametrization of the algebraic family of curves in the abstract necessarily has to contain infinitely many charts (see also \cite{ex}). More precisely, the number of charts in a $0$-mild parametrization will depend on $\epsilon$ and tend to infinity as $\epsilon$ tends to zero. We will construct an explicit $1/\alpha$-mild parametrization of this family where the size of the parametrization does not depend on $\epsilon$. Moreover, we can do this for any family of power-subanalytic curves using a pre-parametrization result of \cite{unif} or \cite{para}. Taking the above example in mind, this seems optimal. More on the field of parametrizations can be found in the survey \cite{ex}.

\section{Preliminaries}
In this short section we fix some notation for multidimensional calculus and revisit the formula for arbitrary derivatives of a function in several variables, also known as the Fa\`a di Bruno formula. 

Let $f \in C^\infty(U,\R)$ be an infinitely differentiable function on an open $U \subset (0,1)^m$. For $\nu \in \N^m$ we denote $|\nu| = \nu_1+\ldots + \nu_m$ and
\[
f^{(\nu)}(x) = \left(\frac{\partial^{|\nu|}}{\partial x_1^{\nu_1}\cdots \partial x_m^{\nu_m}} f\right)(x).
\]
For $f \in C^\infty(U,\R^n)$ we set $f^{(\nu)} = (f_1^{(\nu)},\ldots,f_n^{(\nu)})$. Finally, for $x \in U$ and $\mu \in \R^m$ we denote:
\[
x^\mu = \prod_{i = 1}^m x_i^{\mu_i}.
\]

\begin{definition}[Mild functions] \label{defmild}
Let $f \in C^\infty(U,\R)$, $A,B \in \R_{>0}$ and $C \in \R_{\geq 0}$. We say that $f$ is $(A,B,C)$-mild if for any $x \in U$ and $\nu \in \N^m$:
\[
|f^{(\nu)}(x)| \leq B A^{|\nu|}(|\nu|!)^{C+1}.
\]
A map $f \in C^\infty(U,\R^n)$ is $(A,B,C)$-mild if all its component functions are. We say that $f$ is $C$-mild if it is $(A,B,C)$-mild for some $A$ and $B$.
\end{definition}

A $0$-mild function is analytic, but the converse is not true. The composition of $C$-mild functions is $C$-mild. A proof can be found in the next section. It relies on the formula for arbitrary order derivatives of a composition, which is also known as the Fa\`a di Bruno formula. In the univariate case, in propositions \ref{comp}, \ref{main} and \ref{weakpower}, we will use a more compact version stated in \cite[p.51]{dio}.

\begin{proposition}[Fa\`a di Bruno, \cite{faa}]\label{chain}
Suppose that $n$ is a positive integer, $V \subset \R^d$ and $U \subset \R^e$ are open, $f: V \to \R$, $g: U \to V$ and that $f$ and $g$ are $C^n$. For any $x \in U$ and $\nu \in \N^e$ with $|\nu| = n$ we have that: $$(f \circ g)^{(\nu)}(x) = \sum_{1 \leq |\lambda| \leq n} f^{(\lambda)}(g(x)) \sum_{s = 1}^n \sum_{p_s(\nu,\lambda)} \nu! \prod_{j = 1}^s \frac{(g^{(l_j)}(x))^{k_j}}{k_j!(l_j!)^{k_j}}$$ where $p_s(\nu,\lambda)$ is the set consisting of all $k_1,\ldots,k_s \in \N^d$ with $|k_i| > 0$ and $l_1,\ldots,l_s \in \N^e$ with $0 \prec l_1 \prec \ldots \prec l_s$ such that: $$ \sum_{i = 1}^sk_i = \lambda$$ and $$\sum_{i = 1}^s |k_i|l_i = \nu.$$ Here $l_i \prec l_{i+1}$ means that $|l_i| < |l_{i+1}|$ or, if $|l_i| = |l_{i+1}|$, then $l_i$ comes lexicographically before $l_{i+1}$.
\end{proposition}

\section{The mildness of $P_\alpha$}
In this section we prove that $P_\alpha$ is $1/\alpha$-mild. Writing $P_\alpha$ as a composition of two functions, we will use the Fa\`a di Bruno formula to bound the derivatives. The proof uses the fact that the composition of $0$-mild functions is $0$-mild. Indeed, this follows from the theory of Gevrey functions \cite{Ge}, and a proof in full generality can be found in \cite{para}. To be self contained, we briefly repeat its proof, but formulated in terms of mild functions.

\begin{proposition} \label{comp}
Suppose that $f$ and $g$ are respectively $(A_f,B_f,C)$-mild and $(A_g,B_g,C)$-mild. Then the composition $f \circ g$ is $C$-mild. Moreover, if $C = 0$ we have:
\[
A = A_g(1+A_fB_g) \quad \text{and} \quad B = \frac{A_fB_fB_g}{1+A_fB_g} < B_f.
\]
\end{proposition}

\begin{proof}
By Proposition \ref{chain}, using the triangle inequality and the mildness of $f$ and $g$, we obtain for any $x \in (0,1)$ a sum of positive terms:
\[
|(f \circ g)^{(n)}(x)| \leq \sum_{k = 1}^n B_fA_f^k (k!)^{C+1} \left( \sum_{n,k} B_n(k_1,\ldots,k_n) \prod_{i = 1}^n (B_gA_g^i(i!)^{C+1})^{k_i}\right),
\]
where the sum is over all $n$ and $k$ satisfying the relations stated in the theorem. Since $x \mapsto x^{C+1}$ is convex, taking the $(C+1)$-th rooth of the constants where needed, we are done if we can show that if $C = 0$ the sum can be bounded by $BA^nn!$.

To this end, consider the maps:
\[
\psi(x) = \frac{B_f}{1-A_f(x-B_g)} \quad \text{and} \quad \phi(x) = \frac{B_g}{1 - A_gx}.
\]
One easily verifies that $\psi^{(k)}(B_g) = B_fA_f^kk!$ and $\phi^{(k)}(0) = B_gA_g^kk!$, which are exactly the bounds on the derivatives of $f$ and $g$. Now in this case, one can compute $(\psi \circ \phi)^{(n)}(0)$ directly, without using the chain rule. One checks that:
\[
(\psi \circ \phi)^{(n)}(0) = \frac{A_fB_fB_g}{1+A_fB_g} (A_g(1+A_fB_g))^n n!.
\]
\end{proof}

Clearly, the composition of a $C_1$-mild map and a $C_2$-mild map is $\max(C_1,C_2)$-mild. We need two more elementary lemma's which will be used frequently.

\begin{lemma} \label{umild}
Suppose that $\alpha \geq 1$ and let $u_\alpha$ be the map $(0,1) \to \R$ given by $u_\alpha(x) = 1-x^{-\alpha}$. Then, for any $x \in (0,1)$ and $k > 0 \in \N$, we have:
\[
x^{\alpha+k}|u_\alpha^{(k)}(x)| \leq \alpha^k k!.
\]
\end{lemma}
\begin{proof}
This follows immediately from:
\[
u_\alpha^{(k)}(x) = (-1)^{k+1}x^{-(\alpha+k)}\alpha(\alpha+1)\cdots(\alpha+(k-1))
\]
using that $\alpha \geq 1$ to obtain the desired upper bound.
\end{proof}

Clearly, if $\alpha < 1$, one can find the upper bound $k!$ instead of $\alpha^k k!$ in the lemma.

\begin{lemma} \label{expmild}
For $r,s \in \R_{>0}$ we have:
\[
\max_{x > 0} x^{-r}\exp(-sx^{-\alpha}) = \left(\frac{r}{es\alpha}\right)^{r/\alpha}.
\]
\end{lemma}
\begin{proof}
The maximum is reached in $x = (s\alpha/r)^{1/\alpha}$.
\end{proof}

\begin{proposition} \label{main}
The map $$P_\alpha: (0,1) \to (0,1): x \mapsto \exp(1-x^{-\alpha})$$ is $(A,e,1/\alpha)$-mild, where $A = 6\alpha$ if $\alpha \geq 1$ and $A = 3 (2/\alpha)^{1/\alpha}$ for $0 < \alpha < 1$.
\end{proposition}
\begin{proof}
Defining $u_\alpha(x) = 1-x^{-\alpha}$ as in Lemma \ref{umild}, we will use the Fa\`a di Bruno formula for the map $P_\alpha = \exp \circ \, u_\alpha$. This yields:
\[
P_\alpha^{(n)}(x) = \sum_{k = 1}^n \exp(u_\alpha(x)) \left( \sum_{n,k} B_n(k_1,\ldots,k_n) \prod_{i = 1}^n (u_\alpha^{(i)}(x))^{k_i}\right).
\]
The summation over $n,k$ is subject to the conditions of Proposition \ref{chain}: $\sum_{i = 1}^n ik_i = n$ and  $\sum_{i = 1}^n k_i = k$.
Thus, in the product over the derivatives of $u_\alpha$, we obtain a factor:
\[
\prod_{i = 1}^n (x^{-(\alpha+i)})^{k_i} = x^{-\sum_{i = 1}^n (\alpha+i)k_i} = x^{-(k\alpha+n)}.
\]
Therefore, we obtain:
\begin{align*}
P_\alpha^{(n)}(x) &= \sum_{k = 1}^n x^{-(k\alpha+n)}\exp(u_\alpha(x)) \left( \sum_{n,k} B_n(k_1,\ldots,k_n) \prod_{i = 1}^n (x^{\alpha+i}u_\alpha^{(i)}(x))^{k_i}\right) \\
&=  x^{-n}e^{1-(1/2)x^{-\alpha}}\sum_{k = 1}^n x^{-k\alpha}e^{-(1/2)x^{-\alpha}} \left( \sum_{n,k} B_n(k_1,\ldots,k_n) \prod_{i = 1}^n (x^{\alpha+i}u_\alpha^{(i)}(x))^{k_i}\right).
\end{align*}
We proceed by taking absolute value and using the triangle inequality. Suppose now $\alpha \geq 1$, the other case is similar. By Lemma \ref{umild} we have that:
\[
x^{\alpha+i}|u_\alpha^{(i)}(x)| \leq \alpha^i i!
\]
and by Lemma \ref{expmild} we find that:
\[
x^{-k\alpha}e^{-(1/2)x^{-\alpha}} \leq \left(\frac{2k\alpha}{e\alpha}\right)^k \leq 2^k k!
\]
where we have used the inequality $k ^k \leq e^k k!$. Thus by Proposition \ref{comp} the sum can be bounded by $(3\alpha)^n n!$. Finally, by Lemma \ref{expmild}, we have:
\[
x^{-n}e^{1-(1/2)x^{-\alpha}} \leq e \left(\frac{2n}{e\alpha} \right)^{n/\alpha} \leq e (2/\alpha)^{n/\alpha} (n!)^{1/\alpha},
\]
using that $\alpha \geq 1$ to obtain $(2/\alpha)^{1/\alpha} \leq 2$.
\end{proof}

\section{parametrizations}
In this section we prove some results that indicate that the map $P_\alpha$ can be used to construct mild parametrizations. We start with a classical example that motivates this idea. We generalize this result in Proposition \ref{power}, which yields a parametrization result for definable families of curves. Proposition \ref{weakpower} is more general, but we obtain a $(1/\alpha + 1)$-mild map in that case.

\begin{definition}[Mild parametrization]
Let $X$ be a subset of $\R^n$. A $C$-mild parametrization of $X$ is a finite collection of $C$-mild maps $\{\phi_i: (0,1)^m \to X \mid i = 1,\ldots,N\}$ such that $X$ is covered by the images of these maps. If $X_T$ is a family of subsets of $\R^n$, then a uniform $C$-mild parametrization is a finite collection of families of $C$-mild maps $\{\phi_i: T \times (0,1)^m \to X_T \mid i = 1,\ldots,N\}$. That is: there exist $A$,$B$ and $C$ such that for any $t \in T$ the maps $\phi_{i,t}$ are $(A,B,C)$-mild and $X_t$ is covered by the union of all their images.
\end{definition}

\subsection{Yomdin's example}\label{sec41}
Consider the family of hyperbolas in $(0,1)^2$ given by $xy = \epsilon^2$, with $\epsilon \in (0,1)$. There is no uniform $0$-mild parametrization for this family (see \cite[Section 3.2]{ex}). We will show that it has a $1/\alpha$-mild parametrization consisting of three charts. We will use the map $P_\alpha$ as a power substitution, similarly to the map $x \mapsto x^2$ in \cite[Lemma 4.2]{ex} and more generally $x \to x^r$ in \cite[Proposition 4.1.5]{unif}. 

First, consider the family as the graph of the map $f_\epsilon(x) = f(\epsilon,x) = \epsilon^2/x$ on the interval $\epsilon^2 < x < 1$. Now, on this interval, the first order derivative (with respect to $x$) is not uniformly bounded as $x \to \epsilon^2$. To this end we divide the domain in three parts: $\epsilon^2 < x < \epsilon$, $x = \epsilon$ and $\epsilon < x < 1$. Changing variables, the first case becomes the third case and the second case is trivial. Hence we proceed on the interval $\epsilon < x < 1$.

Let $C$ be the set $\{ (\epsilon,x) \in (0,1)^2 \mid \epsilon < x < 1\}$ and let $\tilde{C}$ be the preimage under the map $P:(0,1)^2 \to (0,1)^2: (\epsilon,x) \mapsto (P_\alpha(\epsilon),P_\alpha(x))$. Then for any $\epsilon$ and $x$ in $\tilde{C}$ we have:
\[
e^{1-\epsilon^{-\alpha}} < e^{1-x^{-\alpha}} < 1 \iff \epsilon < x < 1.
\]
We claim that for any $\epsilon$ the map $(f \circ P)_{\epsilon}$ is $1/\alpha$-mild. Moreover, the constants $A$ and $B$ do not depend on $\epsilon$ (uniform mildness). To this end, note that $$(f \circ P)(\epsilon,x) = e^{2(1-\epsilon^{-\alpha})-(1-x^{-\alpha})}.$$ We use the same approach as in the proof of the mildness of $P_\alpha$. The $n$-th derivative with respect to $x$ can be bounded in absolute value by:
\[
|(f \circ P)_{\epsilon}^{(n)}(x)| \leq \sum_{k = 1}^n e^{2(1-\epsilon^{-\alpha})-(1-x^{-\alpha})} \left( \sum_{n,k} B_n(k_1,\ldots,k_n) \prod_{i = 1}^n (|u_\alpha^{(i)}(x)|)^{k_i}\right).
\]
Writing the exponent as $2(1-\epsilon^{-\alpha}) - 2(1-x^{-\alpha}) + (1-x^{-\alpha})$ and using that $e^{(1-\epsilon^{-\alpha})} < e^{-(1-x^{-\alpha})}$, bounding the above sum completely reduces to the proof of Proposition \ref{main}. (Here we really used that we bounded the first order derivative.) Finally, one should revert the transformation in $\epsilon$ and map $(0,1)$ onto the domain of the obtained function. More precisely, denote  $\Gamma(f_\epsilon)$ for the graph of $f_\epsilon$, then we get:
\[
\phi: (0,1) \to \Gamma(f_\epsilon) : x \mapsto ( (1-P^{-1}_\alpha(\epsilon))x+P^{-1}_\alpha(\epsilon), \epsilon e^{-(1-x^{-\alpha})}).
\]
Note that we do not need the power substitution in $\epsilon$, as we will show in Theorem \ref{curvepara}.

\subsection{Results on compositions}\label{sec3}
We would like to generalize this example, using the results of \cite{unif}. Their pre-parametrization result (Theorem 4.3.1) leads to a uniform $C^r$-parametrization by using the power substitution $P_r: x \to x^r$. More precisely, given a map $\phi$ obtained by this result, the composition $\phi \circ P_r$ is mild (up to order $r$). Using the map $P_\alpha$, it might be possible to construct a uniform $1/\alpha$-mild parametrization, which would improve bounds on the number of rational points of bounded heights on sets that we can parametrize. We will show that $\phi \circ P_\alpha$ is $1/\alpha$-mild. This yields a uniform $1/\alpha$-mild parametrization for power-subanalytic curves.

We start with a more general class of functions defined in \cite{unif}, which also become mild after composing with $P_\alpha$.

\begin{definition}[Weakly mild functions]
Let $f \in C^\infty(U,\R)$, $A,B \in \R_{>0}$ and $C \in \R_{\geq 0}$. We say that $f$ is weakly $(A,B,C)$-mild if for any $x \in U$ and $\nu \in \N^m$:
\[
|f^{|\nu|}(x)| \leq \frac{B A^{|\nu|}(|\nu|!)^{C+1}}{x^\nu}.
\]
\end{definition}

The next result is similar to \cite[Proposition 4.1.5]{unif}. Intuitively, near $0$, $x \to x^r$ is a vanishing function of order $r$, thus is capable of rendering functions mild up to order $r$. In our case, this will be up to order $+\infty$. In the proof of the next proposition, we use that $\alpha \geq 1$, but clearly, it is also true for $0 < \alpha < 1$.

\begin{proposition}\label{weakpower}
Let $f:U \to \R$, $U \subset (0,1)$ open, be weakly $(A,B,0)$-mild and suppose that $f^{(1)}$ is weakly $(A,B,0)$-mild. Then the map $f \circ P_\alpha: P_\alpha^{-1}(U) \to \R$ is $(1+1/\alpha)$-mild.
\end{proposition}


\begin{proof}
Of course, we will use the Fa\`a di Bruno formula and the triangle inequality again. This leaves us with the task to bound:
\[
|(f \circ P_\alpha)^{(n)}(x)| \leq \sum_{k = 1}^n \frac{BA^kk!}{e^{(k-1)(1-x^{-\alpha})}} \left( \sum_{n,k} B_n(k_1,\ldots,k_n) \prod_{i = 1}^n (|P_\alpha^{(i)}(x)|)^{k_i}\right).
\]
Here, we have $k-1$ in the denominator because of the additional assumption on $f^{(1)}$. From the proof of Proposition \ref{main}, we know that:
\begin{align*}
|P_\alpha^{(i)}(x)| &= \left|\sum_{l = 1}^i x^{-(l\alpha+i)}e^{1-x^{-\alpha}} \left( \sum_{i,l} B_i(l_1,\ldots,l_i) \prod_{j = 1}^i (x^{\alpha+j}u_\alpha^{(j)}(x))^{l_j}\right)\right| \\
&\leq x^{-(\alpha+1)i}e^{1-x^{-\alpha}}\sum_{l = 1}^i \left( \sum_{i,l} B_i(l_1,\ldots,l_i) \prod_{j = 1}^i (x^{\alpha+j}|u_\alpha^{(j)}(x)|)^{l_j}\right).
\end{align*}
The sum over $l$ can be considered as the composition of an $(1,1,0)$-mild map and a $(\alpha,1,0)$-mild map, yielding a $(2\alpha,1,0)$-mild map. Using the relations for the sum over $n,k$ and Lemma \ref{expmild} we get:
\begin{align*}
|(f \circ P_\alpha)^{(n)}(x)| &\leq \sum_{k = 1}^n \frac{BA^kk!}{e^{(k-1)(1-x^{-\alpha})}}x^{-(\alpha+1)n}e^{k(1-x^{-\alpha})} \left( \sum_{n,k} B_n(k_1,\ldots,k_n) \prod_{i = 1}^n ((2\alpha)^ii!)^{k_i}\right) \\
&=  x^{-(\alpha+1)n}e^{1-x^{-\alpha}}\sum_{k = 1}^n BA^kk!\left( \sum_{n,k} B_n(k_1,\ldots,k_n) \prod_{i = 1}^n ((2\alpha)^ii!)^{k_i}\right) \\
&\leq e\left(\frac{(\alpha+1)n}{e\alpha}\right)^{(\alpha+1)n/\alpha} B(2\alpha(A+1))^nn! \\
&\leq eB\left(\frac{\alpha+1}{\alpha}\right)^{(\alpha+1)n/\alpha}(2\alpha(A+1))^n (n!)^{2+1/\alpha}.
\end{align*}
\end{proof}

Following the strategy of the proof of \cite[Proposition 4.1.5]{unif} for several variables, will also yield this result for weakly mild functions in more variables. Now this result is not very satisfying and it seems that it cannot be improved in general. For $\alpha = 1$ it is equal to the bound in \cite[Lemma 77]{ccs}.

However, if $f$ is of a specific form, we can improve it. The maps we consider are the maps obtained from the pre-parametrization result of \cite{unif}. Our motivating example falls within this scope. The key observation is that these maps have a good interaction with $P_\alpha$. Indeed, the main obstruction to a better result in the proof above is that we completely factor out a power of $x$ in the estimate of $|P^{(i)}_\alpha(x)|$. We will not have to do this below. Let us be more precise about which maps are given by this parametrization result.

\begin{definition}[a-b-m functions]\label{abmdef}
A map $b:U \subset (0,1)^m \to \R$ is bounded-monomial if it is either identically zero, or is of the form $x^\mu$ for some $\mu \in \R^m$ and its image is bounded. A map $b: U \to \R^n$ is bounded-monomial if all of its component functions are. A map $f: U \to \R$ is analytic-bounded-monomial (a-b-m) if $f$ is of the from
\[
f(x) = b_j(x)F(b(x))
\]
where $b$ is a bounded-monomial map, $b_j$ a component function of $b$ and $F$ is a unit on $\im(b)$, that is: it is analytic and non-vanishing on an open neighborhood of the topological closure of $\im(b)$.
\end{definition}

It is shown in \cite[Corollary 4.2.4]{unif} that if all partial derivatives of the associated bounded-monomial map $b$ of $f$ are bounded, then $f$ and all its first order derivatives are weakly $0$-mild. Thus one could apply the previous proposition to this class of functions. In the next proposition, we use that $\alpha \geq 1$ in the proof and the statement, but it is straightforward to deduce an analogue for $0 < \alpha < 1$. One just has to compute the constant $A$.

\begin{proposition} \label{power}
Suppose that $b(x) = x^{\mu}$ is a bounded-monomial map whose first order derivatives are bounded on an open $U \subset (0,1)^m$. Let $P: (0,1)^m \to (0,1)^m$ be given by
\[
P(x_1,\ldots,x_m) = (P_\alpha(x_1),\ldots,P_\alpha(x_m)).
\]
Then $b \circ P: P^{-1}(U) \to \R$ is $(A,B,1/\alpha)$-mild with $A = 2\alpha(2mN+1)$ with $N = \max_{i = 1}^m |\mu_i|$, $B = e^{|\mu|}M^2$ and $M$ a constant that depends on $b$.
\end{proposition}

\begin{proof}
We use the same strategy as in the proof of Proposition \ref{main}: using the multivariate Fa\`a di Bruno formula. We write 
\[
(b \circ P)(x) = e^{\sum_{i = 1}^m\mu_i(1-x_i^{-\alpha})} = (E \circ u_\alpha)(x_1,\ldots,x_m)
\]
where $E = \exp(x_1+\ldots+x_m)$ and $u_\alpha = (u^1_\alpha(x_1),\ldots,u^m_\alpha(x_m))$ where $u^i_\alpha(x_i) = \mu_i(1-x_i^{-\alpha})$. We may suppose that all $\mu_i$ are nonzero, otherwise the map does not depend on $x_i$ and any derivative with respect to $x_i$ is trivial. We have that:
\[
(E \circ u_\alpha)^{(\nu)}(x) = \sum_{1 \leq |\lambda| \leq n} E^{(\lambda)}(u_\alpha(x)) \sum_{s = 1}^n \sum_{p_s(\nu,\lambda)} \nu! \prod_{j = 1}^s \frac{(u_\alpha^{(l_j)}(x))^{k_j}}{k_j!(l_j!)^{k_j}}.
\]
As in the proof of Proposition \ref{main}, we rewrite the formula as follows:
\[
(E \circ u_\alpha)^{(\nu)}(x) = \sum_{1 \leq |\lambda| \leq n} \frac{1}{x^{\nu+\alpha\lambda}}E^{(\lambda)}(u_\alpha(x)) \sum_{s = 1}^n \sum_{p_s(\nu,\lambda)} \nu! x^{\nu+\alpha\lambda}\prod_{j = 1}^s \frac{(u_\alpha^{(l_j)}(x))^{k_j}}{k_j!(l_j!)^{k_j}},
\]
which equals
\[
e^{|\mu|}\frac{1}{x^\nu}e^{-(1/2)\sum_{i = 1}^m \mu_ix_i^{-\alpha}}  \sum_{1 \leq |\lambda| \leq n} \frac{1}{x^{\alpha\lambda}}e^{-(1/2)\sum_{i = 1}^m \mu_ix_i^{-\alpha}}  \sum_{s = 1}^n \sum_{p_s(\nu,\lambda)} \nu! x^{\nu+\alpha\lambda}\prod_{j = 1}^s \frac{(u_\alpha^{(l_j)}(x))^{k_j}}{k_j!(l_j!)^{k_j}}.
\]
Taking the extra factor $x^{\nu+\alpha\lambda}$ into account, using the relations on $l_j$ and $k_j$, we can bound everything concerning $u_\alpha$ by the appropriate bounds of an $(\alpha,N,0)$-mild function (see Lemma \ref{umild}). Now let $I$ be such that $x_I = \min_{i = 1}^m x_i$. Let $\mu' \in \R^m$ be given by $\mu'_i = \mu_i$ for $i \neq I$ and $\mu'_I = \mu_I - 1$. Then we have:
\[
\frac{1}{x^{\alpha\lambda}}e^{-(1/2)\sum_{i = 1}^m \mu_ix_i^{-\alpha}} \leq \frac{1}{x_I^{\alpha|\lambda|}}e^{-(1/2)x_I^{-\alpha}}e^{-(1/2)\sum_{i = 1}^m \mu'_ix_i^{-\alpha}} \leq \left(\frac{2|\lambda|}{e}\right)^{|\lambda|} e^{-(1/2)\sum_{i = 1}^m \mu'_ix_i^{-\alpha}}
\]
by Lemma \ref{expmild}. Now denote $K = e^{-(1/2)\sum_{i = 1}^m \mu'_ix_i^{-\alpha}}$, then $e^{|\mu'|}K^2 = ((\partial / \partial x_I)(b)/\mu_I) \circ P$. By our assumption on the partial derivatives of $b$, taking the maximum, this can be bounded independent of $I$. Thus we find that $K < M$ for some $M$ depending on $b$ only.

Similarly we find that
\[
\frac{1}{x^\nu}e^{-(1/2)\sum_{i = 1}^m \mu_ix_i^{-\alpha}} \leq \left(\frac{2|\nu|}{e\alpha}\right)^{|\nu|/\alpha}M.
\]
Using the multivariate version of Proposition \ref{comp} (see \cite[Corollary 2.5.1]{para}) we obtain that for any $x \in U$:
\[
|(b \circ P)^{(\nu)}(x)| \leq e^{|\mu|}M^2 (2/\alpha)^{|\nu|/\alpha}(\alpha(1+m2N))^{|\nu|} (|\nu|!)^{1+1/\alpha}.
\]
\end{proof}

As we have mentioned in the example, the fact that all first order partial derivatives of $b$ are bounded is crucial. Not only for the proof, but the statement can fail if it is not bounded. Indeed, if one does not bound the first order derivative in the example of Section \ref{sec41}, one would obtain the relation $e^{2(1-\epsilon^{-\alpha})} < e^{1-x^{-\alpha}}$, which is not sufficient to bound $e^{-2(1-x^{-\alpha})}$.

\begin{theorem}\label{abmmild}
Suppose that $f: U \to \R$ is a-b-m on an open $U \subset (0,1)^m$ such that the first order derivatives of its associated bounded-monomial map $b$ are bounded. Then $f \circ P: P^{-1}(U) \to \R$ is $1/\alpha$-mild.
\end{theorem}

\begin{proof}
This follows by Proposition \ref{power} and the fact that the unit $F$ is $0$-mild since it is analytic on an open neighborhood of $\overline{\text{Im}(b)}$ \cite[Proposition 2.2.10]{primer}. This concludes the proof by the remark below Proposition \ref{comp} and the fact the the product of $C$-mild functions is $C$-mild.
\end{proof}


Before moving to our final result, let us point out that by the method above, one can show that the parametrization of the exponential-algebraic surface on the last line of \cite[p. 507]{countpointsexp} is $1/g$-mild (improving $1+1/g$). This improves the upper bound on rational points in that paper to $c(\epsilon)(\log T)^{35+\epsilon}$.

We will now use the pre-parametrization result of \cite[Theorem 3.5]{para}, which is a slight modification of \cite[Theorem 4.3.1]{unif}, to obtain a uniform $1/\alpha$-mild parametrization of a family of power-subanalytic curves. This improves the $2$-mild parametrization of subanalytic curves of \cite{ccs}.

\begin{theorem}\label{curvepara}
If $X_T$ is a power-subanalytic family of power-subanalytic curves in $(0,1)^n$, then it has a uniform $1/\alpha$-mild parametrization.
\end{theorem}

\begin{proof}
By \cite[Theorem 3.5]{para} we obtain finitely many $f: C \subset T \times (0,1) \to [-1,1]^{n-1}$, where $C$ is an open cell in $T \times (0,1)$. The map $f$ is of the following form
\[
f(x) = b_j(t,x) F(b(t,x)),
\]
where $b_j$ is a component function of $b$, all component functions of $b$ are of the form $a(t)x^r$ for some definable function $a(t)$, $r \in \R$ and $F$ is a non-vanishing analytic function on an open neighborhood of $\overline{\im(b)}$. Moreover, $(\partial/\partial x)(b)$ is (uniformly) bounded. Consider the map
\[
P: T \times (0,1) \to T \times (0,1) : (t,x) \mapsto (t,P_\alpha(x))
\]
and set $\tilde{C} = P^{-1}(C)$. It follows that for any $t$ the map $(f \circ P)_t$ is $(A,B,1/\alpha)$-mild, where $A$ and $B$ are independent of $t$ (using the same technique as in the proof of Proposition \ref{power}, using the fact that $a(t)x^r$ and its first order derivative are uniformly bounded). Since in the case of curves the walls of the cell bounding the variable $x$ only depend on $t$, for any $t$ they are automatically $(1,1,0)$-mild.
\end{proof}

It would be interesting to know whether the techniques in this paper can be applied to the cellular maps in \cite{ccs} to obtain a uniform $1/\alpha$-mild parametrization of families of subanalytic sets (of arbitrary dimension). For power-subanalytic sets, one does not really require an analogue of \cite[Proposition 4.2.6]{unif}, as mentioned in the introduction of this section. One could argue by induction instead. The obstacle in that case is the fact that $P_\alpha^{-1}(C)$ is not power-subanalytic.

\bibliographystyle{alpha}
\bibliography{Mildlib}{}

\begin{thebibliography}{CPW20}

\bibitem[BN19]{ccs}
Gal Binyamini and Dmitry Novikov.
\newblock Complex cellular structures.
\newblock {\em Ann. of Math. (2)}, 190(1):145--248, 2019.

\bibitem[CPW20]{unif}
R.~Cluckers, J.~Pila, and A.J. Wilkie.
\newblock Uniform parameterizations of subanalytic sets and diophantine
  applications.
\newblock {\em Ann. Sci. Ecole Norm. Sup}, 53(1):1--42, 2020.

\bibitem[CS96]{faa}
G.~M. Constantine and T.~H. Savits.
\newblock A multivariate {F}a\`a di {B}runo formula with applications.
\newblock {\em Trans. Amer. Math. Soc.}, 348(2):503--520, 1996.

\bibitem[Gev18]{Ge}
Maurice Gevrey.
\newblock Sur la nature analytique des solutions des \'{e}quations aux
  d\'{e}riv\'{e}es partielles. {P}remier m\'{e}moire.
\newblock {\em Ann. Sci. \'{E}cole Norm. Sup. (3)}, 35:129--190, 1918.

\bibitem[Hil19]{para}
S.~Van Hille.
\newblock Smooth parameterizations of power-subanalytic sets and compositions
  of {G}evrey functions.
\newblock {\em arXiv:1905.06408}, 2019.

\bibitem[JMT11]{mildpara}
G.~O. Jones, D.~J. Miller, and M.~E.~M. Thomas.
\newblock Mildness and the density of rational points on certain transcendental
  curves.
\newblock {\em Notre Dame J. Form. Log.}, 52(1):67--74, 2011.

\bibitem[JW15]{dio}
G.~O. Jones and A.~J. Wilkie, editors.
\newblock {\em O-minimality and diophantine geometry}, volume 421 of {\em
  London Mathematical Society Lecture Note Series}.
\newblock Cambridge University Press, Cambridge, 2015.
\newblock Lecture notes from the LMS-EPSRC course held at the University of
  Manchester, Manchester, July 2013.

\bibitem[KP02]{primer}
Steven~G. Krantz and Harold~R. Parks.
\newblock {\em A primer of real analytic functions}.
\newblock Birkh\"{a}user Advanced Texts: Basler Lehrb\"{u}cher. [Birkh\"{a}user
  Advanced Texts: Basel Textbooks]. Birkh\"{a}user Boston, Inc., Boston, MA,
  second edition, 2002.

\bibitem[Pil10]{countpointsexp}
Jonathan Pila.
\newblock Counting rational points on a certain exponential-algebraic surface.
\newblock {\em Ann. Inst. Fourier (Grenoble)}, 60(2):489--514, 2010.

\bibitem[Yom08]{yomdin2}
Y.~Yomdin.
\newblock Analytic reparametrization of semi-algebraic sets.
\newblock {\em J. Complexity}, 24(1):54--76, 2008.

\bibitem[Yom15]{ex}
Y.~Yomdin.
\newblock Smooth parametrizations in dynamics, analysis, diophantine and
  computational geometry.
\newblock {\em Jpn. J. Ind. Appl. Math.}, 32(2):411--435, 2015.

\end{thebibliography}
\end{document}